\documentclass[11pt,reqno]{amsart}
\usepackage{tikz}
\usetikzlibrary{shapes.geometric}
\usepackage[center]{caption}
\usepackage{xcolor}

\usepackage{amssymb,latexsym,amsmath,mathtools}
\usepackage{graphicx,mathrsfs}
\usepackage{subfigure}
\usepackage{float}
\usepackage{hyperref,url}
\hypersetup{colorlinks=true, urlcolor=black, linkcolor=black, citecolor=black}
\usepackage{pict2e}
\usepackage{amstext}
\usepackage{bbm}
\usepackage{amsthm}
\usepackage{dsfont}
\usepackage{amsfonts}
\usepackage{comment}

\newcommand{\sub}{\subseteq}
\newcommand{\Z}{\mathbb{Z}}
\newcommand{\R}{\mathbb{R}}

\newcommand{\N}{\mathbb{N}}
\newcommand{\eps}{\varepsilon}

\newcommand{\fj}[2]{\left\lfloor #1 \right\rfloor_#2}

\newtheorem{theorem}{Theorem}[section]
\newtheorem{lemma}[theorem]{Lemma}

\newtheorem{proposition}[theorem]{Proposition}

\newtheorem{definition}[theorem]{Definition}

\numberwithin{equation}{section}
\DeclarePairedDelimiter{\norm}{\lVert}{\rVert}

\makeatletter
\let\oldnorm\norm
\def\norm{\@ifstar{\oldnorm}{\oldnorm*}}

\makeatother

\begin{document}

\title[]
{Improved packing of hypersurfaces in $\R^d$}



\author{Xianghong Chen}

\address[Xianghong Chen]{Department of Mathematics, Sun Yat-sen University, Guangzhou, Guangdong 510275, P.R. China}
\email{chenxiangh@mail.sysu.edu.cn}

\author{Tongou Yang}

\address[Tongou Yang]{Department of Mathematics, University of California\\
Los Angeles, California 90095, United States}
\email{tongouyang@math.ucla.edu}

\author{Yue Zhong}

\address[Yue Zhong]{ Department of Mathematics, Sun Yat-sen University \\
Guangzhou, Guangdong 510275, P.R. China}
\email{zhongy69@mail3.sysu.edu.cn}



\begin{abstract}
For $d\ge 1$, we construct a compact subset $K\sub \R^{d+1}$ containing a $d$-sphere of every radius between $1$ and $2$, such that for every $\delta\in (0,1)$, the $\delta$-neighbourhood of $K$ has Lebesgue measure $\lesssim |\log \delta|^{-2/d}$. This is the smallest possible order when $d=2$, and improves a result of Kolasa-Wolff \cite{Kolasa_Wolff}. Our construction also generalises to H\"{o}lder-continuous families of $C^{2,\alpha}$ hypersurfaces with nonzero Gaussian curvature.
\end{abstract}

\maketitle

\tableofcontents

\section{Introduction}

In this paper, we establish the following main result.

\begin{theorem}\label{thm:sphere}
    Let $d\ge 1$. Then for every $\delta\in (0,1)$, there exists a compact subset $K_\delta\sub \R^{d+1}$ containing a $d$-sphere of every radius between $1$ and $2$, such that the $\delta$-neighbourhood of $K_\delta$ has Lebesgue measure $\lesssim_d|\log \delta|^{-2/d}$.
\end{theorem}
The first sentence in the abstract then follows from Theorem \ref{thm:sphere} via a standard iterative argument.

\subsection{Background}\label{sec:literature}
Theorem \ref{thm:sphere} is part of a subject known as the construction of Kakeya-type sets. Roughly speaking, it is concerned with packing geometric objects in $\R^d$ into a thin subset.

\subsubsection{Packing lines}
Packing straight line segments in $\R^d$ is known as the (linear) Kakeya problem. Besicovitch \cite{Besicovitch} constructed a compact subset of $\R^2$ of Lebesgue measure $0$ that contains a unit line segment in every direction, and a set of this type is now known as a Besicovitch/Kakeya set. On the other hand, Davies \cite{DaviesKakeya} showed that every Kakeya set in $\R^2$ must have full Hausdorff dimension (and hence Minkowski dimension) $2$. In higher dimensions $d\ge 3$, one can easily check that $K\times [0,1]^{d-2}$ is a Kakeya set in $\R^d$ when $K$ is a Kakeya set in $\R^2$, but it remains a famous open problem whether every Kakeya set in $\R^d$ must have full Hausdorff dimension. Some progress in this open problem in recent years include \cite{KatzZahlKayeka, KakeyaRdZahl, KakeyaRdHRZ, WangZahlSticky, WangZahlAssouad} (and see references therein).

Even in $\R^2$, it is interesting to find an exact dimension function for Kakeya sets. Schoenberg \cite{Schoenberg1, Schoenberg2} (see also \cite{Keich}) constructed a Kakeya set $K_1$ satisfying $|N_\delta(K_1)|\lesssim |\log \delta|^{-1}$. On the other hand, by proving a sharp upper bound for the Kakeya maximal function, C\'ordoba \cite{CordobaKakeyalowerbound} (see also references therein) showed that for every Kakeya set $K\sub\R^2$, we must have $|N_\delta(K)|\gtrsim |\log \delta|^{-1}$. This means that both results of Schoenberg and C\'ordoba are the best possible (up to absolute constant factors).\footnote{The case of Hausdorff dimension functions is more difficult; for example, Chen-Yan-Zhong \cite{ChenYanZhong} (see also \cite{Keich}) constructed a Kakeya set whose sharp Hausdorff gauge function is between $\delta^2 |\log \delta|$ and $\delta^2|\log \delta| (\log |\log \delta|)^\eps$, $\eps>0$.} For a few earlier results on line packing in $\R^2$, see \cite{Perron,Besicovitch1963}.

\subsubsection{Packing curves/surfaces}
Variants of the linear Kakeya problem have been extensively studied. For example, Kolasa-Wolff \cite{Kolasa_Wolff} constructed for $d\ge 1$ a compact subset $K_2\sub \R^{d+1}$ of Lebesgue measure $0$ that contains a $d$-sphere of every radius between $1$ and $2$. By inspecting their construction, one can deduce that
$|N_\delta(K_2)|\lesssim_d |\log \delta|^{-2/d}(\log |\log \delta|)^{2/d}$. Thus our main Theorem \ref{thm:sphere} is a refinement of \cite{Kolasa_Wolff} by removing the $\log\log$ factor.

This refinement is nontrivial, especially in the case $d=2$, for the following reason. In \cite{Kolasa_Wolff}, through the corresponding maximal operator bounds, Kolasa-Wolff proved that for every compact set $K\sub \R^{d+1}$ containing a $d$-sphere of every radius between $1$ and $2$, we must have $|N_\delta(K)|\gtrsim_d |\log \delta|^{-1}$ if $d\ge 2$. If $d=1$, an analogous but weaker result holds, namely, $|N_\delta(K)|\gtrsim_\eps \delta^{\eps}$ for every $\eps>0$, which is proved in \cite{Wolff_Kakeya_L3}\footnote{See also \cite{PYZ2022}, where the method could slightly refine the lower bound $\delta^\eps$ by carrying out the technical computations.}. In particular, when $d=2$, Theorem \ref{thm:sphere} attains the lower bound $|\log \delta|^{-1}$ for $|N_\delta(K)|$ (up to a constant factor). It remains open whether $|\log \delta|^{-2}$ is sharp for $d=1$, and whether $|\log \delta|^{-2/d}$ (or $|\log \delta|^{-1}$) is sharp for $d\ge 3$. 

See also \cite{BesicovitchRado, Kinney, Davies, Talagrand, Sawyer, Wisewell, ChangCsornyei, MR4534746} for results on packing of curves/surfaces in different settings.

\subsubsection{Lines versus circles}
The linear Kakeya set $K_1$ constructed by Schoenberg \cite{Schoenberg1, Schoenberg2} obeys $|N_\delta(K_1)|\lesssim |\log \delta|^{-1}$ which is sharp. The curved Kakeya set $K$ constructed in Theorem \ref{thm:sphere} obeys $|N_\delta(K)|\lesssim |\log \delta|^{-2}$, which is significantly smaller than $|\log \delta|^{-1}$. This means that it is easier for geometric objects with nonzero curvature to be packed in a thinner subset.

\subsection{Generalisation to hypersurfaces}
In fact, we will consider packing of more general hypersurfaces with nonzero Gaussian curvature, with regularity as low as possible. The case of spheres will follow as a corollary. We need a little more notation to state the general result.

Let $d\ge 1$. We say that $R_0\sub \R^d$ is an (axis-parallel) rectangle if $R_0=\prod_{i=1}^d I_i$ where each $I_i$ is a compact interval of positive length.

\begin{definition}\label{defn:regular_curved}
Let $[a_0,a_0+\delta_0]\sub \R$ and $R_0\sub \R^d$ be a compact rectangle. We say that $f(a,x)$
is a regular curved function defined in 
$[a_0,a_0+\delta_0]\times R_0$ if $\norm{\partial_a f}_{\infty}$, $\norm {\nabla_x \partial_a f}_{\infty}$, $\norm{D^2_x f}_\infty$ (and hence $\norm{f}_\infty$ and $\norm {\nabla_x f}_\infty$) are all finite, and in addition there are constants $c_0>0$, $\alpha\in (0,1]$, $C_\alpha>0$, $\eta\in (0,1]$ 
and $A_\eta>0$ such that the following holds 
for all $a,a'\in [a_0,a_0+\delta_0]$ and $x,x'\in R_0$:
\begin{align}
& |\det D^2_x f(a,x)|\ge c_0 , \label{eqn:curved} \\
& |D^2_x f(a,x)-D^2_x f(a,x')|\le C_\alpha|x-x'|^\alpha, \label{eqn:Lipschitz}\\
& |\partial_a f(a,x)-\partial_a f(a',x)|\le A_\eta |a-a'|^\eta,\label{eqn:Holder1}\\
&|\nabla_x f (a,x)-\nabla_x  f(a',x)|\le A_\eta|a-a'|^\eta.\label{eqn:Holder2}
\end{align}
\end{definition}
The first condition \eqref{eqn:curved} is the key curvature assumption. The other three relations are mild regularity conditions. In particular, if $f$ is $C^3$ in $x$, then $\alpha=1$.

\begin{theorem}\label{thm:f(a,x)_0}
    Assume that $R_0=\prod_{i=1}^d I_i\sub \R^d$ is a compact rectangle and $f(a,x)$ is a regular curved function defined in
$[a_0,a_0+\delta_0]\times R_0$. Then for every $\delta\in (0,1)$, there exists a compact subset $K\sub \R^{d+1}$ that satisfies the following.
\begin{enumerate}
    \item For every $a\in [a_0,a_0+\delta_0]$, the set $K$ contains a translated copy of the graph of $x\mapsto f(a,x)$ over $R_0$.
    \item $|N_\delta(K)|\lesssim |\log \delta|^{-(1+\alpha)/d}$. Here the implicit constant is independent of $a_0,\delta_0,\delta$, but depends on
\begin{equation}\label{eqn:vec_Cf}
\begin{aligned}
    \vec C_f
    &:=(d,|I_1|,\dots,|I_d|,c_0,\alpha, C_\alpha,\eta,A_\eta,\\
    &\norm{\partial_a f}_{\infty},\norm{\nabla_x f}_\infty,\norm {\nabla_x \partial_a f}_{\infty}, \norm{D^2_x f}_\infty).
\end{aligned}
\end{equation}
\end{enumerate}
\end{theorem}
Theorem \ref{thm:f(a,x)_0} generalises \cite[Theorem 1.3]{YangZhong} in the following aspects. First, it contains a higher dimensional result. Second, it considers a more general family of functions $f(a,x)$, not just functions of the form $af(x)$. Third, in the assumptions, we weakened the boundedness of $f'''$ to just $f''$ being Lipschitz (when $\alpha=1$) and even $\alpha$-H\"older continuous (whence the result is also weakened when $\alpha<1$.) Lastly and more importantly, we managed to remove the unnecessary condition (1.9) in \cite{YangZhong}. Previously, we thought it was necessary since it was used to prove \cite[Lemma 2.1]{YangZhong}, which was needed to show some monotonicity property related to the tangency of curves. It turns out, however, that this monotonicity property is not necessary in the proof of the upper bound of measure.

{\it Remark.} It is interesting to consider the following special case for which Theorem \ref{thm:f(a,x)_0} applies: $f(a,x):=ax+x^2$, where $a\in [0,1]$ and $x\in [0,1]$. Applying the bilipschitz mapping $(x,y)\mapsto (x,y-x^2)$ applied to $K$, it implies that there exists a family of unit line segments in $\R^2$ whose union $K'$ satisfies $|N_\delta(K')|\lesssim |\log \delta|^{-2}$, which seems to contradict C\'ordoba's result \cite{CordobaKakeyalowerbound} mentioned earlier. However, taking a closer look at the construction in Section \ref{sec:f(a,x)}, one can see that $N_\delta(K)$ only contains parabolas with apertures contained in a small interval of length $\sim\delta$, and so $N_\delta(K')$ only contains line segments with slopes contained in a small interval of length $\sim\delta$.

Next, we further generalise Theorem \ref{thm:f(a,x)_0} to hypersurfaces other than graphs, which is necessary in the proof of Theorem \ref{thm:sphere}.
\begin{definition}\label{defn:regular_curved_surface}
Let $I_0\times P_0\sub \R\times\R^{d+1}$ be a compact subset. We say that $\Phi(a,x):I_0\times P_0\to \R$ defines a regular family of curved hypersurfaces if $\norm{\partial_a \Phi}_{\infty}$, $\norm {\nabla_x \partial_a \Phi}_{\infty}$, $\norm{D^2_x \Phi}_\infty$ (and hence $\norm{\Phi}_\infty$ and $\norm {\nabla_x \Phi}_\infty$) are all finite, and in addition there are constants $c_0>0$, $c'_0>0$, $\alpha\in (0,1]$, $C_\alpha>0$, $\eta\in (0,1]$ and $A_\eta>0$ such that the following holds 
for all $a,a'\in  I_0$ and $x,x'\in P_0$:
\begin{align}
& \{x\in \,\mathrm{interior\,\, of} \,\,P_0:\Phi(a,x)=0 \}\ne\varnothing, \label{eqn:nonempty}\\
&|\nabla_x \Phi(a,x)|\ge c_0',\label{eqn:gradient}\\
& \left|\det \begin{bmatrix}
        D^2_x \Phi & \nabla_x \Phi\\
        (\nabla_x \Phi)^T & 0
    \end{bmatrix}(a,x)\right|\ge c_0, \label{eqn:curved'} \\
& |D^2_x \Phi(a,x)-D^2_x \Phi(a,x')|\le C_\alpha|x-x'|^\alpha, \label{eqn:Lipschitz'}\\
& |\partial_a \Phi(a,x)-\partial_a \Phi(a',x)|\le A_\eta |a-a'|^\eta,\label{eqn:Holder1'}\\
&|\nabla_x \Phi (a,x)-\nabla_x  \Phi(a',x)|\le A_\eta|a-a'|^\eta.\label{eqn:Holder2'}
\end{align}
\end{definition}
Note that \eqref{eqn:nonempty}, \eqref{eqn:gradient} and \eqref{eqn:Lipschitz'} imply that for each $a\in I_0$, the equation $\Phi(a,x)=0$ defines a (compact piece of a) $C^{2,\alpha}$ hypersurface $S_a$ in $\R^{d+1}$. The condition \eqref{eqn:curved'} means that the Gaussian curvatures of $S_a$, $a\in I_0$ are uniformly away from $0$.

\begin{theorem}\label{thm:surface}
Let $I_0\times P_0\sub \R\times\R^{d+1}$ be a compact subset and let $\Phi(a,x):I_0\times P_0\to \R$ define a regular family of curved hypersurfaces $S_a$. Then for every $\delta\in (0,1)$, there exists a compact subset $K\sub\R^{d+1}$ that contains a translated copy of $S_a$ for every $a\in I_0$, such that $|N_\delta(K)|\lesssim |\log \delta|^{-(1+\alpha)/d}$. The implicit constant here does not depend on $\delta$, but may depend on every other parameter involved.
\end{theorem}
For the case of $d$-spheres, it is easy to see that $\Phi(a,x):=|x|-a$, $1\le a\le 2$ satisfies Definition \ref{defn:regular_curved_surface} with $\alpha=1$ and every other parameter involved depending on $d$ only. Thus Theorem \ref{thm:sphere} follows from Theorem \ref{thm:surface}.

We also remark that Theorem \ref{thm:surface} assumes no convexity of the surfaces $S_a$ as long as the Gaussian curvatures are away from $0$. In particular, it applies to the family $\Phi(a,x_1,x_2,x_3):=x_3-ax_1x_2$ of hyperbololic paraboloids in $\R^3$. It is an interesting question to study, for example, the case of light cones given by $\Phi(a,x_1,x_2,x_3):=x_3^2-ax_1^2-ax_2^2$, which is not covered by Theorem \ref{thm:surface}. Lastly, we remark that having nonzero Gaussian curvature is not necessary for the hypersurfaces to be able to be packed into a thin subset, as can be seen from the family of cylinders $\{(a\cos \theta,a\sin \theta):\theta\in [0,2\pi]\}\times [0,1]\sub \R^3$, $1\le a\le 2$. Indeed, this is essentially the packing of circles in $\R^2$, in which we are even able to compress better than in the case of spheres in $\R^3$. We therefore anticipate that it is the abundance of normal directions of the hypersurfaces that prevents packing.

\subsubsection{Cinematic curvature}
It is very natural to compare the curvature condition \eqref{eqn:curved'} to another widely studied notion of curvature, known as the {\it cinematic curvature}, which was introduced by Sogge \cite{Sogge_first_cinematic}, and later further studied in \cite{Kolasa_Wolff, PYZ2022, ChenGuoYang2023, Zahl2023}. 

We point out the main difference. Roughly speaking, a cinematic curvature condition measures how much the curvatures of a family change according to the given parameters, and formally it is usually concerned with third-order derivatives of the functions. This type of condition is usually used to prove a $L^p$-$L^q$ upper bound for its associated maximal operator (such as \cite{Kolasa_Wolff} mentioned in Section \ref{sec:literature}), which consequently {\it prevents} the packing of curves into a thin subset.

In comparison, the curvature condition as in \eqref{eqn:curved'} is less sophisticated, and as stated in Theorems \ref{thm:f(a,x)_0} and \ref{thm:surface}, it enables us to pack curved surfaces into thinner subset.

In particular, for the family of circles in $\R^2$ with radii ranging from $1$ to $2$ (and also the family of parabolas with aperture ranging through $1$ and $2$ studied in \cite{YangZhong}), it can be checked that both the cinematic curvature condition and \eqref{eqn:curved'} hold, and this is why we have both upper and lower bounds for the measure of packing of curves.

\subsection{Notation}

\begin{enumerate}
 \item We use the standard notation $a=O_C(b)$, or $|a|\lesssim_C b$ to mean there is a constant $C$ such that $|a|\leq Cb$. In many cases, such as results related to polynomials, the constant $C$ will be taken to be absolute or depend on the polynomial degree only, and so we may simply write $a=O(b)$ or $|a|\lesssim b$. Similarly, we define $\gtrsim$ and $\sim$.

    \item For a set $S\sub \R^d$ and $\delta>0$, we denote by $N_\delta(S)$ the $\delta$-neighbourhood of $S$. $|A|$ will denote the Lebesgue measure of $A$, where the dimension will be clear from the context.
    
\end{enumerate}

\subsection{Acknowledgements}
Xianghong Chen and Yue Zhong are supported in part by the National Key R\&D Program of China (No. 2022YFA1005700) and the NNSF of China (No. 12371105). Tongou Yang is supported by the Croucher Fellowships for Postdoctoral Research. The authors would like to thank Professor Lixin Yan for helpful suggestions.

\section{Packing graphs of functions}\label{sec:f(a,x)}
In this section we prove the following proposition, which implies Theorem \ref{thm:f(a,x)_0}.

\begin{proposition}\label{prop:f(a,x)_delta}
    
Assume that $R_0$ is a compact rectangle and $f(a,x)$ is a regular curved function defined in
$[a_0,a_0+\delta_0]\times R_0$. 
Then for every $M\ge 1$, there exist functions $u(a)$, $v(a)$ of the order $O(\delta_0 2^{-M^d/3})$ that are piecewise constant on intervals of the form $(a_0+j2^{-M^d}\delta_0,a_0+(j+1)2^{-M^d}\delta_0)$, $0\le j\le 2^{M^d}-1$, such that the $\delta_0 2^{-M^d}$ neighbourhood of the following subset
\begin{equation}\label{eqn:F_M}
\begin{aligned}
    F_M:=&\{(x+u(a),f(a,x)+v(a)):x\in R_0,a\in [a_0,a_0+\delta_0]\}
\end{aligned}
\end{equation}
has Lebesgue measure $\lesssim \delta_0 M^{-1-\alpha}$.

Here and throughout this section, the implicit constants are independent of $M,a_0,\delta_0$, but depend on $\vec C_f$ as in \eqref{eqn:vec_Cf}.
\end{proposition}

The rest of this section is devoted to the proof of this proposition, which closely resembles that of \cite[Section 2]{YangZhong}. Apart from the improvements mentioned right after Theorem \ref{thm:f(a,x)_0}, we also use a minor technique (see Section \ref{sec:repeat_xj}) so that we will no longer need to remove a small portion of the interval.

Now we start with the proof. We can assume $R_0=[0,1]^d$ for simplicity of notation, since the implicit constants are allowed to depend on $|I_1|,\dots,|I_d|$. We can also assume $M$ is even, and is large enough depending on $\vec C_f$, since for small $M$ the
result is trivial.

\subsection{Notation}
Start with the initial ``curved rectangle"
\begin{equation*}
    \{(x,f(a,x)):x\in R_0,a\in [a_0,a_0+\delta_0]\}.
\end{equation*}
Fix $M\in 2\N$. Denote
\begin{equation*}
    a_n:=a_0+n\delta_0 2^{-M^d}, \quad n=0,1,\dots,2^{M^d}-1,
\end{equation*}
and correspondingly divide the initial curved rectangle into $2^{M^d}$ smaller curved rectangles
\begin{equation*}
    T_{0,n}:=\{(x,f(a,x)):x\in R_0,a\in [a_n,a_n+\delta_0 2^{-M^d}]\},\quad n=0,1,\dots,2^{M^d}-1.
\end{equation*}
Define a ``representative'' function of $T_{0,n}$ by
\begin{equation*}
    L_{0,n}(x):=f(a_n,x).
\end{equation*}
Also, define the following scales
\begin{align}
    &\lambda:=\text{smallest integer }\ge  2\log_2 M,\label{eqn:defn_lambda}\\
    &m:=\left(\frac M 2-1\right)^d+\lambda.\label{eqn:defn_m}
\end{align}
\subsubsection{A continuous path on grid points}\label{sec:path}
Tile $R_0=[0,1]^d$ uniformly into a grid of $(M/2)^d$ congruent rectangles. This gives rise to exactly $(M/2-1)^d$ ``interior" grid points. 

Denote $J:=M/2-1$. We define a ``continuous path" $\tilde x_j$, $1\le j\le J^d$ on these interior grid points. More precisely, $\tilde x_j$ is a bijection from $1\le j\le J^d$ to the set of all such grid points, such that $|\tilde x_{j+1}-\tilde x_j|$ is exactly the distance between two adjacent grid points. 

There are many possibilities of such continuous paths. For example, when $d=2$, we may define a ``back-and-forth" path as follows: 
\begin{equation*}
    \tilde x_j:=\begin{cases}
        (2jM^{-1},2M^{-1}),\quad & 1\le j\le J,\\
        (2(M-j-1)M^{-1},4M^{-1}),\quad & J+1\le j\le 2J,\\
        \text{etc.},
    \end{cases}
\end{equation*}
until we exhaust all such grid points. When $d=3$ (denoted as $xyz$ space), a back-and-forth path may be defined as follows: first, we define a continuous path in the plane $z=2M^{-1}$, which ends at some point $(a,b,2M^{-1})$. Then we define the next point on the path by $(a,b,4M^{-1})$, from which we then define another continuous path in the plane $z=4M^{-1}$, and then continue this process until we exhaust all grid points. The general construction follows by induction on the dimension $d$.

\subsubsection{A slightly modified sequence of grid points}\label{sec:repeat_xj}
Let $\tilde x_j$ be a continuous path on the interior grid points given by Section \ref{sec:path}. For technical purposes (see \eqref{eqn:single_thickness_bound}), we will be using a slightly different sequence. More precisely, we define a sequence $x_j$, $1\le j\le m$ as follows (recall \eqref{eqn:defn_lambda}):
\begin{equation}\label{eqn:defn_x_j}
x_j:=
    \begin{cases}
        \tilde x_1,\quad  & 1\le j\le \lambda,\\
        \tilde x_{j-\lambda}, \quad  &\lambda+1\le j\le m.
    \end{cases}
\end{equation}
That is, the first grid point $\tilde x_1$ is repeated by exactly $\lambda+1$ times.

\subsection{The construction algorithm}
\subsubsection{Step 1}
For each $n=1,3,\dots,2^{M^d}-1$, we apply a suitable translation to $T_{0,n}$ so that the function $L_{0,n}$, after translation, will be tangent to $L_{0,n-1}$ at $x_1$. That is, we need to find translations $u:=u_{1,n}$, $v:=v_{1,n}$ such that
\begin{align}    
    &f(a_n, x_1-u) + v = f(a_{n-1}, x_1), \label{eqn:equal_x1}\\
    &\nabla_x f (a_n, x_1-u) =  \nabla_x f(a_{n-1}, x_1). \label{eqn:tangent_x1}
\end{align}
To find the expressions of $u,v$, we first focus on \eqref{eqn:tangent_x1}. First, using the implicit function theorem and the fact that $|\det D^2_{x} f|\ge c_0$, we see that for $M$ large enough, 
such $u$ always exists, and we have the bound
\begin{equation*}
|u|\lesssim a_n-a_{n-1}=\delta_0 2^{-M^d}.
\end{equation*}
To find the expression of $u$, we use Taylor expansion and \eqref{eqn:Lipschitz} and \eqref{eqn:Holder2} to compute for each $1\le l\le d$ that (where $\partial_l$ denotes the partial derivative with respect to the $l$-th coordinate of $x$)
\begin{align*}
&\partial_l f(a_n,x_1-u)=\partial_l f(a_n,x_1)-u\cdot \nabla_x \partial_{l} f(a_n,x_1)+O(\delta_0 2^{-M^d})^2,\\
&\partial_l f(a_{n-1},x_1)=\partial_l f(a_n,x_1)-(a_n-a_{n-1})\partial_{a}\partial_l f(a_n,x_1)+O(\delta_0 2^{-M^d})^{1+\eta},
\end{align*}
and so \eqref{eqn:tangent_x1} gives that
\begin{equation}\label{eqn:Dec_23_u1}
u=u_{1,n}=\delta_0 2^{-M^d}h(a_n,x_1)+O(\delta_0 2^{-M^d})^{1+\eta},
\end{equation}
where
\begin{equation}\label{eqn:defn_h}
    h(a,x):=(D^2_x f)^{-1}(a,x)\nabla_x \partial_a f(a,x).
\end{equation}
Plugging this into \eqref{eqn:equal_x1}, we can find the expression of $v$.

After Step 1, we obtain $2^{M^d-1}$ ``larger'' curved figures
\begin{equation}\label{eqn:T1}
    T_{1,n}:=T_{0,n}\bigcup (T_{0,n+1}+(u_{1,n+1},v_{1,n+1})),\quad n=0,2,\dots,2^{M^d}-2,
\end{equation}
which are well compressed near $x_1$. Denote
\begin{equation*}
    T_1:=\bigcup_{n\in 2\N} T_{1,n}.
\end{equation*}

\subsubsection{Step $2$}
We continue in a similar way, this time compressing $T_{1,n}$, $n=0,2,\dots,2^{M^d}-2$ at $x_2$. More precisely, for $n=2,6,10,\dots,2^{M^d}-2$, we translate $T_{1,n}$ further by some $(u_{2,n},v_{2,n})$ so that the curve $L_{0,n}$, after translation, is tangent to $L_{0,n-2}$ at $x_2$. Similarly to \eqref{eqn:T1}, we obtain $2^{M^d-2}$ larger curved figures $T_{2,n}$, $n=0,4,\dots,2^{M^d}-4$, which are well compressed at $x_2$. Denote
\begin{equation*}
    T_2:=\bigcup_{n\in 4\N} T_{2,n}.
\end{equation*}

\subsubsection{Step $j$}
Given a general $1\le j\le m$. For each $n$ of the form $2^j p+2^{j-1}$, $p=0,1,\dots,2^{M^d-j}-1$, we translate $T_{j-1,n}$ by some $(u_{j,n},v_{j,n})$ so that the curve $L_{0,n}$ is tangent to $L_{0,n-2^{j-1}}$ at $x_j$. By the same computation as in Step 1, we have, similar to \eqref{eqn:Dec_23_u1},
\begin{equation}\label{eqn:unj}
    u_{j,n} = \delta_0 2^{j-1-M^d}h(a_n,x_j)+O(\delta_0 2^{j-1-M^d})^{1+\eta}.
\end{equation}
Now we analyse the expression of 
\begin{equation*}
    v_{j,n}:=f(a_{n-2^{j-1}},x_j)-f(a_n,x_j-u_{j,n}).
\end{equation*}
By Taylor expansion and \eqref{eqn:curved} through \eqref{eqn:Holder2}, we have
\begin{align*}
    f(a_{n-2^{j-1}},x_j)=f(a_n,x_j)-\delta_0 2^{j-1-M^d}\partial_a f(a_n,x_j)+O(\delta_0 2^{j-1-M^d})^{1+\eta},
\end{align*}
and using \eqref{eqn:unj},
\begin{align*}
    &f(a_n,x_j-u_{j,n})\\
    &=f(a_n,x_j)-u_{j,n}\cdot \nabla_x f(a_n,x_j)+O(\delta_0 2^{j-1-M^d})^{2}\\
    &=f(a_n,x_j)-\delta_0 2^{j-1-M^d} h(a_n,x_j)\cdot \nabla_x f(a_n,x_j)+O(\delta_0 2^{j-1-M^d})^{1+\eta},
\end{align*}
where $h$ is as defined in \eqref{eqn:defn_h}. Thus we have
\begin{equation}\label{eqn:vnj}
    \begin{aligned}
        v_{j,n}
        &=\delta_0 2^{j-1-M^d} h(a_n,x_j)\cdot \nabla_x f(a_n,x_j)-\partial_a f(a_n,x_j)+O(\delta_0 2^{j-1-M^d})^{1+\eta},
    \end{aligned}
\end{equation}
and so in particular,
\begin{equation}\label{eqn:Dec_9_02}
    |u_{j,n}|\lesssim \delta_0 2^{j-M^d},\quad |v_{j,n}|\lesssim \delta_0 2^{j-M^d}.
\end{equation}
Hence, similarly to \eqref{eqn:T1}, we obtain $2^{M^d-j}$ larger curved figures $T_{j,n}$, $n=0,2^j,\dots,2^{M^d}-2^j$, which are well compressed at $x_j$. The representative function of such a $T_{j,n}$ is still $L_{0,n}$. Denote
\begin{equation*}
    T_j:=\bigcup_{n\in 2^j \N} T_{j,n}.
\end{equation*}

\subsubsection{End of construction}
We perform the above procedures for $m$ times at each tangent point $x_j$, $j=1,\dots,m$, arriving at the set $T_{m}$. 

\subsection{Computation of translations}

We now analyse the sum of all translations $(u_{j,n},v_{j,n})$ that have been performed to the original curved rectangle $T_{0,n}$ at Steps $j=1,2,\dots,m$.

Given $n=0,1,\dots,2^{M^d}-1$, by binary expansions, we know there exist unique integers $\eps_j=\eps_j(n)\in \{0,1\}$, $1\le j\le m$ such that $n=\sum_{j=1}^{m} \eps_j 2^{j-1}$. 

For convenience, we introduce the notation
\begin{equation*}
    \fj{n}{j}:=n-(n\,\,\mathrm{ mod  }\,\,2^{j-1}),
\end{equation*}
which means the ``integral part" of $n$ in in $2^{j-1}\Z$. Then we note that $\eps_j(n)=1$ if and only if $u_{\fj n j}$, $v_{\fj n j}$ are defined by Step $j$. 

\begin{proposition}\label{prop:translation_bound}
For $1\le j\le m$ and $n=0,1,\dots,2^{M^d}-1$, denote the partial sums 
\begin{equation*}
    U_{j,n}:=\sum_{i=1}^j \eps_i(n) u_{i,\fj n i},\quad V_{j,n}:=\sum_{i=1}^j \eps_i(n) v_{i,\fj n i}.
\end{equation*}
Then we have
\begin{equation}\label{eqn:partial_sum_bound}
    |U_{j,n}|\lesssim \delta_0 2^{j-M^d},\quad |V_{j,n}|\lesssim  \delta_0 2^{j-M^d}.
\end{equation}
In particular, the total translation $(U_{m,n},V_{m,n})$ of $T_{0,n}$ after $m$ steps satisfies
    \begin{equation}\label{eqn:unvn}
        |U_{m,n}|\lesssim \delta_0 2^{m-M^d},\quad |V_{m,n}|\lesssim \delta_0 2^{m-M^d}.
    \end{equation}
\end{proposition}
\begin{proof}
    The proof is obtained by inspection using \eqref{eqn:Dec_9_02}. For example, if $m>100$ and $n=27$, then $T_{0,n}$ is translated according to the representative curves of $T_{0,27}$, $T_{0,26}$, $T_{0,24}$ and $T_{0,16}$ in steps $1,2,4,5$, respectively; it remains unchanged in all other steps. Note that $\eps_j(27)=1$ if and only if $j=1,2,4,5$, whence $27-(27\,\,\mathrm{ mod }\,\,2^{j-1})=27,26,24,16$, respectively.
\end{proof}
In this way, we see that if we define $u(a)=U_{m,n}$ and $v(a)=V_{m,n}$ for $a\in (a_n,a_n+\delta_0 2^{-M^d}])$, then the set $F_M$ defined in \eqref{eqn:F_M} is just $T_m$, since by \eqref{eqn:defn_lambda} and \eqref{eqn:defn_m}, we have in particular that $M^d-m\ge M^d/3$. Thus this proposition ensures that for a large $M$, the total distance of translations is tiny; in particular, the projections of the translated curved rectangles onto the $x$ coordinate are within the $O(\delta_0 2^{-M^d/3})$ neighbourhood of $R_0$.

\subsection{Upper bound of measure}
In this subsection, we control the measure of the $\delta_0 2^{-M^d}$ neighbourhood of $F_M$. Since $2^{-M^d}\ll M^{-1-\alpha}$ for $M$ large enough, it suffices to show that for each $x'\in N_{\delta_0 2^{-M^d}}(R_0)$,
\begin{equation}\label{eqn:measure_bound}
    |\{y:(x',y)\in F_M\}|\lesssim \delta_0 M^{-1-\alpha}.
\end{equation}
Since $2^{-M^d/3}\ll M^{-1-\alpha}$ for $M$ large enough, it suffices further to prove \eqref{eqn:measure_bound} for those $x'\in R_0$ that are at least $O(\delta_0 2^{-M^d/3})$ away from the boundary of $R_0$. For these $x'$, we can ensure that $x'-U_{j,n}$ is within $R_0$ for each $j,n$.

Now we fix such an $x'$. Then there is some $j$ such that the grid point $x_j$ is closest to $x'$; we may assume $j\ge \lambda$ by \eqref{eqn:defn_x_j}. For each $n=0,1,\dots,2^{M^d}-1$, we write
\begin{equation}\label{eqn:n_sum}
    n=p+\sum_{i=1}^j \eps_i \cdot 2^{i-1}\quad (p=0,2^j,\cdots, 2^{M^d}-2^j).
\end{equation}
In other words, for each $j$, we combine the translated curved rectangles $T_{0,n}+(U_{m,n},V_{m,n})$ into $2^{M^d-j}$ groups, and the curved rectangles in the $p2^{-j}$-th group continue to be translated together after Step $j$.

Denote by $L_{j,n}$ the representative function of $T_{0,n}$ after $j$ steps of translations, namely,
\begin{equation}\label{eqn:Lnj}
    L_{j,n}(x):=L_{0,n}(x-U_{j,n})+V_{j,n},
\end{equation}
where we recall $L_{0,n}(x)=f(a_n,x)$. By the triangle inequality, it suffices to show that the thickness of the $p2^{-j}$-th group is $\lesssim \delta_0 M^{-1-\alpha}2^{j-M^d}$. More precisely, we need to show
\begin{equation*}
    |L_{m,n}(x')-L_{m,p}(x')|+\delta_0 2^{-M^d}\lesssim \delta_0 M^{-1-\alpha}2^{j-M^d}.
\end{equation*}
Here, $|L_{m,n}(x')-L_{m,p}(x')|$ is the distance between the representative functions after $m$ steps of translations, and $\delta_0 2^{-M^d}$ is essentially the thickness of one smallest curved rectangle thicken by $\delta_0 2^{-M^d}$, since $\norm{\partial_a f}_\infty\lesssim 1$. But by our choice that $j\ge \lambda\sim 2\log M$, we have 
\begin{equation}\label{eqn:single_thickness_bound}
    2^{-M^d}\lesssim M^{-1-\alpha}2^{j-M^d},
\end{equation}
and thus our task reduces to showing
\begin{equation*}
    |L_{m,n}(x')-L_{m,p}(x')|\lesssim \delta_0 M^{-1-\alpha}2^{j-M^d}.
\end{equation*}
To this end, we trace back to the configuration right after Step $j$. That is, we let 
\begin{equation*}
    x:=x'-\sum_{i=j+1}^{m-1}u_{i,\fj n i}, 
\end{equation*}
whose distance from $x_j$ is within $O(M^{-1}+2^{j-M^d})=O(M^{-1})$, by the definition that $x_j$ is the closest grid point to $x'$. Using the fact that the translations for the $n$-th and the $p$-th pieces after Step $j$ are the same, we have
\begin{equation*}
    L_{m,n}(x')-L_{m,p}(x')=L_{j,n}(x)-L_{j,p}(x).
\end{equation*}
By \eqref{eqn:Lnj}, we have 
$$
\left\{
\begin{aligned}
    & L_{j,n}(x) =  f(a_n, x- U_{j,n} ) + V_{j,n}, \\
    & L_{j,p}(x) =  f(a_p, x- U_{j,p} ) + V_{j,p}, \\
\end{aligned}
\right.
$$
where, by definition, we have $U_{j,p} = V_{j,p} =0$. Thus
\begin{equation*}
    L_{j,n}(x)-L_{j,p}(x)=f(a_n, x- U_{j,n} )-f(a_p,x)+V_{j,n}.
\end{equation*}
Using Taylor expansion, we have
\begin{align*}
    f(a_n, x- U_{j,n} )
    &=f(a_n,x)-\nabla_x f(a_n,x)U_{j,n}+O(|U_{j,n}|)^2,
\end{align*}
and also
\begin{align*}
    f(a_p,x)=f(a_n,x)+(a_p-a_n)\partial_a f(a_n,x)+O(|a_p-a_n|^2).
\end{align*}
Since $a_n-a_p\le\delta_0 2^{j-M^d}$, by \eqref{eqn:partial_sum_bound}, we have 
\begin{equation*}
\begin{aligned}
    &L_{j,n}(x)-L_{j,p}(x)\\
    &=(a_n-a_p)\partial_a f(a_n,x)-\nabla_x f(a_n,x)\cdot U_{j,n}+V_{j,n}+O(\delta_0^2 2^{2j-2M^d}).
\end{aligned}    
\end{equation*}
Since $2^{2j-2M^d}\ll M^{-1-\alpha}2^{j-M^d}$, it now suffices to show that
\begin{equation}\label{eqn:Dec_25}
    (a_n-a_p)\partial_a f(a_n,x)-\nabla_x f(a_n,x)\cdot U_{j,n}+V_{j,n}=O(\delta_0 M^{-1-\alpha}2^{j-M^d}).
\end{equation}
To this end, we use Proposition \ref{prop:translation_bound}, \eqref{eqn:unj}, \eqref{eqn:vnj} and \eqref{eqn:n_sum} to compute
\begin{align}
    &(a_n-a_p)\partial_a f(a_n,x)-\nabla_x f(a_n,x)\cdot U_{j,n}+V_{j,n}\nonumber\\
    &=\frac{\delta_0}{2^{M^d}}\sum_{i=1}^j       \eps_i 2^{i-1} \left[ \partial_a f(a_n,x) 
    -\partial_a f(a_{\fj n i},x_i)\right.\nonumber\\
    &\left. +h(a_{\fj n i},x_i) \cdot \left(\nabla_x f(a_n,x)-\nabla_x f(a_{\fj n i},x_i)\right)+O(\delta_0 2^{i-M^d})^{\eta}\right],\label{eqn:Dec_10_02}
\end{align}
where $h$ is as in \eqref{eqn:defn_h}. We now fix $i$ and invoke the assumptions in Definition \ref{defn:regular_curved} to estimate
\begin{equation*}
    \partial_a f(a_n,x_i)-\partial_a f(a_{\fj n i}, x_i).
\end{equation*}
Using $n-\fj n i\le 2^{i-1}$, we have $|a_n-a_{\fj n i}|\le \delta_0 2^{i-1-M^d}$. By \eqref{eqn:Holder1}, we have
\begin{equation*}
    |\partial_a f(a_n,x_i)-\partial_a f(a_{\fj n i}, x_i)|\lesssim \delta_0 2^{(i-M^d)\eta}.
\end{equation*}
This means in \eqref{eqn:Dec_10_02}, we can replace $\partial_a f(a_{n}, x_i)$ by $\partial_a f(a_{\fj n i}, x_i)$. In a similar way using \eqref{eqn:Holder2} this time, we may also replace $\nabla_x f(a_n,x)$ by $\nabla_x f(a_{\fj n i},x)$. Thus the right hand side of \eqref{eqn:Dec_10_02} becomes
\begin{equation*}
    O(\delta_0 2^{j-1-M^d})^{1+\eta}+\delta_0 2^{-M^d}\sum_{i=1}^j \eps_i 2^{i-1}[g(x)-g(x_i)],
\end{equation*}
where
\begin{equation*}
\begin{aligned}
    g(x)
    &:=\partial_a f(a_{\fj n i},x) 
    -h(a_{\fj n i},x_i) \cdot \nabla_x f(a_{\fj n i},x)\\
    &=\partial_a f(a_{\fj n i},x) 
    -\left((D^2_x f)^{-1}(a_{\fj n i},x_i)\nabla_x \partial_a f(a_{\fj n i},x_i) \right)\cdot \nabla_x f(a_{\fj n i},x).
\end{aligned}
\end{equation*}
Since $2^{(j-M^d)\eta}\ll M^{-1-\alpha}$, to prove \eqref{eqn:Dec_25}, it suffices to prove
\begin{equation}
    \sum_{i=1}^j 2^i |g(x)-g(x_i)|\lesssim 2^j M^{-1-\alpha}.
\end{equation}
Now to prove this, note that by direct computation, 
\begin{equation*}
\begin{aligned}
    \nabla_x g(x)
    &=\nabla_x \partial_a f(a_{\fj n i},x)-D^2_x f(a_{\fj n i},x) (D^2_x f)^{-1}(a_{\fj n i},x_i)\nabla_x \partial_a f(a_{\fj n i},x_i),
\end{aligned}
    \end{equation*}
which satisfies $\nabla_x g(x_i)=0$. Thus, by \eqref{eqn:Lipschitz}, we have $|g(x)-g(x_i)|\lesssim |x-x_i|^{1+\alpha}$, and so it suffices now to prove the following elementary inequality:
\begin{equation}\label{eqn:elementary}
    \sum_{i=1}^j 2^i |x-x_i|^{1+\alpha}\lesssim 2^j M^{-1-\alpha},\quad \forall \, |x-x_j|\lesssim M^{-1}.
\end{equation}
To prove this, we first perform some reductions. The first reduction is to reduce to essentially the case $\lambda=0$. More precisely, write the above sum as
\begin{equation*}
    \sum_{i=1}^{\lambda} 2^i |x-x_1|^{1+\alpha}+\sum_{i=\lambda+1}^j 2^i |x-x_i|^{1+\alpha}.
\end{equation*}
The first sum is essentially $2^\lambda |x-x_1|^{1+\alpha}$. We need to show $2^{j-\lambda}\gtrsim M^{1+\alpha}|x-x_1|^{1+\alpha}$. But $|x-x_j|\lesssim M^{-1}$, so by definition of the back-and-forth order, we have $|x-x_1|\lesssim (j-\lambda+1)M^{-1}$, and so $|x-x_1|^{1+\alpha}\lesssim 2^{j-\lambda}M^{-1-\alpha}$. Thus it suffices to prove 
\begin{equation}\label{eqn:Dec_25_02}
    \sum_{i=\lambda+1}^j 2^i |x-x_i|^{1+\alpha}\lesssim 2^j M^{-1-\alpha}\quad \forall \, |x-x_j|\lesssim M^{-1}.
\end{equation}
To prove the last inequality, we may perform a second reduction, namely, the special case $x=x_j$. Indeed, for $i=j$, we have $2^j |x-x_j|^{1+\alpha}\lesssim 2^j M^{-1-\alpha}$ since $|x-x_j|\lesssim M^{-1}$. For $i<j$, we have $|x-x_i|\lesssim |x_j-x_i|$, and so to prove \eqref{eqn:Dec_25_02}, it suffices to prove
\begin{equation*}
    \sum_{i=\lambda+1}^{j-1} 2^i |x_j-x_i|^{1+\alpha}\lesssim 2^j M^{-1-\alpha}.
\end{equation*}
Changing variables $j':=j-\lambda$ and $i':=i-\lambda$ and using the definition of $x_j$ \eqref{eqn:defn_x_j}, it is equivalent to proving
\begin{equation*}
    \sum_{i'=1}^{j'-1}2^{i'}|\tilde x_{j'}-\tilde x_{i'}|^{1+\alpha} \lesssim 2^{j'}M^{-1-\alpha}.
\end{equation*}
But by a trivial rescaling and $\alpha\in (0,1]$, this just follows from Lemma \ref{lem:elementary} below. In conclusion, we have proved that $N_{\delta_0 2^{-M^d}}(F_M)\lesssim \delta_0 M^{-1-\alpha}$.

\subsection{An elementary inequality}
\begin{lemma}\label{lem:elementary}
    Let $J,d\in \N$ and let $n_j$, $1\le j\le J^d$ be a sequence in $\{1,\dots,J\}^d$ obeying $|n_{j+1}-n_j|=1$ for all $j$. Then for each $1\le j\le J^d$ we have the estimate
    \begin{equation}\label{eqn:elementary_general}
        \sum_{i=1}^{j}2^i |n_j-n_i|^2\lesssim 2^j,
    \end{equation}
    where the implicit constant is independent of $d,J$.
\end{lemma}
\begin{proof}
Denote $a_i:=|n_j-n_i|^2$. Using summation by parts formula and the fact that $a_j=0$, we have
\begin{equation*}
    \sum_{i=1}^j 2^i a_i=-2a_1-\sum_{i=2}^{j} 2^i(a_{i}-a_{i-1}).
\end{equation*}
Since $a_1=|n_j-n_1|^2\le (j-1)^2\lesssim 2^j$, it suffices to show that
\begin{equation*}
    \sum_{i=2}^{j} 2^i(a_{i}-a_{i-1})=O(2^j),
\end{equation*}
namely,
\begin{equation*}
    \sum_{i=2}^j 2^i \big(|n_j-n_i|+|n_{j}-n_{i-1}|\big)\big||n_j-n_i|-|n_{j}-n_{i-1}|\big|\lesssim 2^j.
\end{equation*}
But since $|n_{j+1}-n_j|=1$, by the triangle inequality, we have
\begin{equation*}
    \big||n_j-n_i|-|n_{j}-n_{i-1}|\big|\le |n_i-n_{i+1}|=1,
\end{equation*}
and so it suffices to show that 
\begin{equation*}
    \sum_{i=2}^j 2^i \big(|n_j-n_i|+|n_{j}-n_{i-1}|\big)\lesssim 2^j.
\end{equation*}
But this follows from a similar summation by parts technique as above. Thus the result follows.
\end{proof}
{\it Remark.} By induction, we can even show that \eqref{eqn:elementary_general} holds with $|n_j-n_i|^2$ replaced by $|n_j-n_i|^\beta$ for any $\beta\in \R$, with the implicit constant depending on $\beta$.

\section{Packing curved hypersurfaces}
In this section we prove the following proposition, which implies Theorem \ref{thm:surface}.

\begin{proposition}\label{prop:surface_delta}
Let $I_0:=[A,B]$ and $I_0\times P_0\sub \R\times\R^{d+1}$ be a compact rectangle. Let $\Phi(a,x):I_0\times P_0\to \R$ define a regular family of curved hypersurfaces $S_a$. Then for every $M\ge 1$, there exists a function $w(a)\in \R^{d+1}$ of the order $O(2^{-cM^d})$ (where $c\sim 1)$ that is piecewise constant on intervals of the form $(A+j2^{-M^d}(B-A),A+(j+1)2^{-M^d}(B-A))$, $0\le j\le 2^{M^d}-1$, such that the $2^{-M^d}$ neighbourhood of the following subset
\begin{equation}\label{eqn:surface_M}
\begin{aligned}
    E:=&\bigcup_{a\in I_0}\{x+w(a):x\in S_a\}
\end{aligned}
\end{equation}
has Lebesgue measure $\lesssim M^{-1-\alpha}$.

Here and throughout this section, the implicit constants are independent of $M$, but may depend on every other parameter involved.
\end{proposition}
The rest of this section is devoted to the proof of this proposition.

\subsection{Partition of unity}
To prove Proposition \ref{prop:surface_delta}, we may first partition $I_0$ into subintervals that are short enough, and prove the proposition with $I_0$ replaced by a shorter subinterval. By abuse of notation we may still denote a shorter subinterval by $I_0$. We may assume without loss of generality that $I_0=[0,1]$.

Now we fix $b\in I_0$. Using \eqref{eqn:gradient} and a partition of unity, there exist $N=O(1)$ many coordinate charts, on each of which the surface $S_b$ is locally a graph of a function. More precisely, there exists a boundedly overlapping cover of $\mathbb S^{d}$ by $N$ open subsets $\Omega$, so that $\R^{d+1}$ is covered by sectors 
\begin{equation*}
    \Gamma:=\{r\omega:\omega\in \Omega,r\ge 0\},
\end{equation*}
and such that for each $\Omega$ the following holds: there exists a rotation $\rho:\R^{d+1}\to \R^{d+1}$, a compact rectangle $R\sub \R^d$, and a function $f_b:R\to \R$ such that (where $\overline{A}$ denotes the closure of $A$)
\begin{equation}\label{eqn:rotation_graph}
    \rho(S_b\cap \overline{\Gamma})=\{(x,f_b(x)):x\in R\}.
\end{equation}
Moreover, since $I_0$ is short enough, by \eqref{eqn:Holder1'} and \eqref{eqn:Holder2'}, we can ensure that for every $a\in I_0$, for the same $\Omega$, $\rho$ and $R$, there exists a function $f_a:R\to \R$ such that \eqref{eqn:rotation_graph} holds with $b$ replaced by $a$. Furthermore, if we define $f(a,x):=f_a(x)$ for $a\in I_0$, $x\in R$, then one can check that $f$ will be a regular curved function as in Definition \ref{defn:regular_curved}, using \eqref{eqn:gradient} through \eqref{eqn:Holder2'}.

We now denote these sectors as $\Gamma_i$ and these functions as $f_i(a,x):I_0\times R_i$, $1\le i\le N$.

Finally, the scale $M$ comes into play. We may assume the scale $M\in N^{1/d}\N$, and denote
\begin{equation}
    M_0:=N^{-1/d}M.
\end{equation}

Now we are ready to present the main construction. See Figure \ref{fig:circle} for the case of circles ($d=1$) after 3 steps.

\begin{figure}
    \centering
    \includegraphics[width=0.45\linewidth]{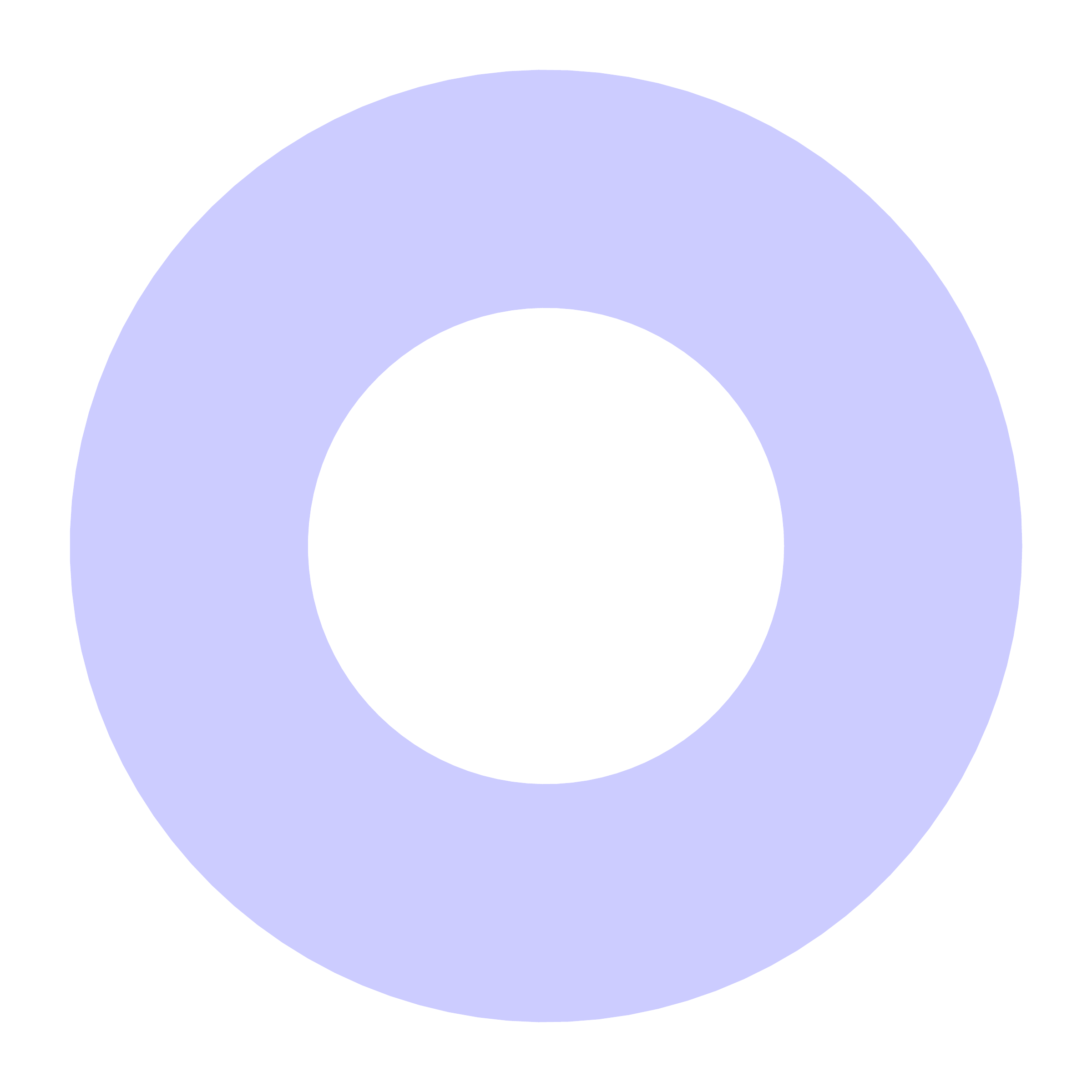}
    \includegraphics[width=0.45\linewidth]{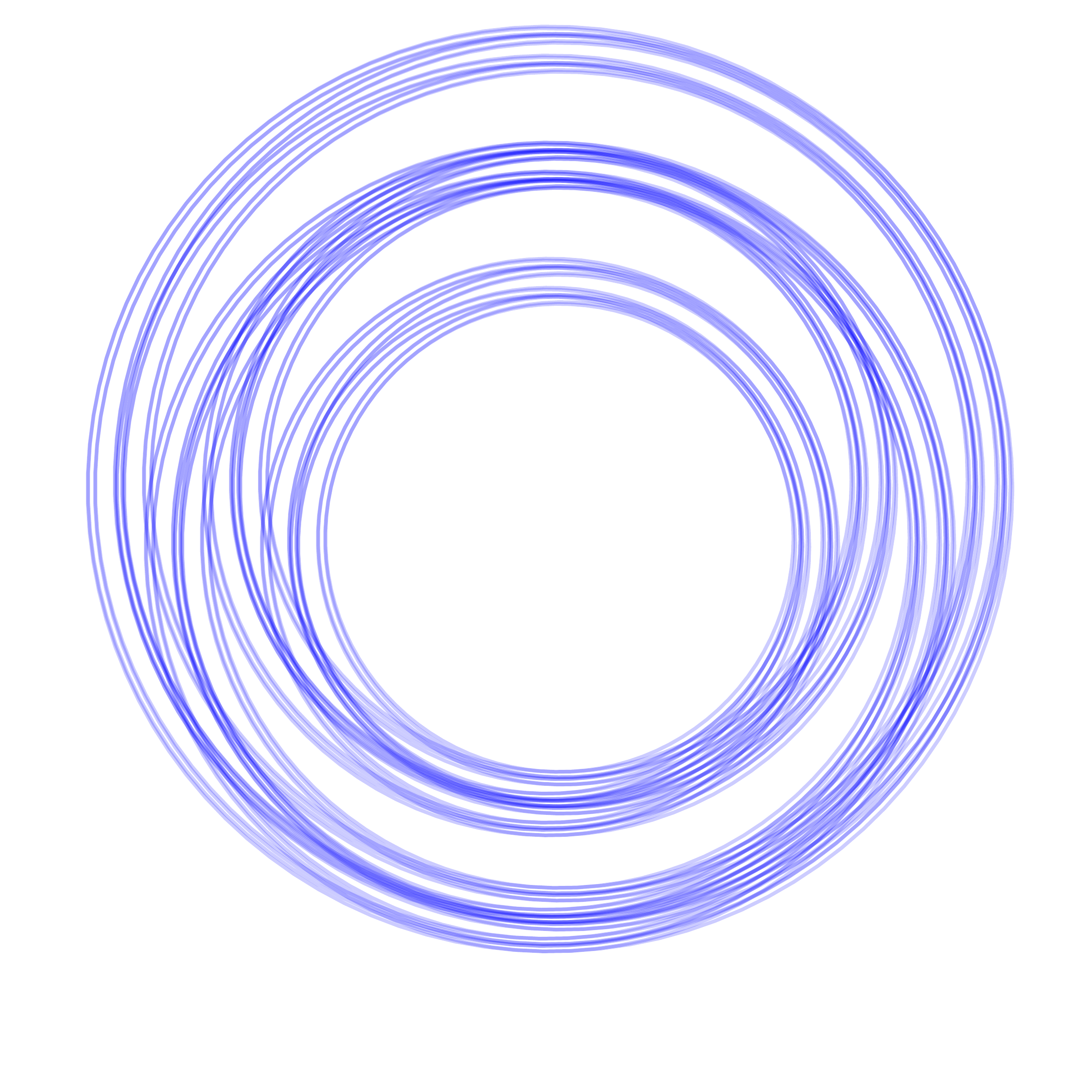}
    \caption{Translations of the annuli after 3 steps}
    \label{fig:circle}
\end{figure}

\subsection{Step 1}
Recall that we assumed $I_0=[0,1]$. We apply Proposition \ref{prop:f(a,x)_delta} to $f_1(a,x)$ with $a_0=0$, $\delta_0=1$ and $M$ replaced by $M_0$. This effectively gives a translation function $w_1(a)\in \R^{d+1}$ of the order $O(2^{-M_0^d/3})\ll M_0^{-1-\alpha}$ that is constant on intervals $I_1$ of the form $(k2^{-M_0^d},(k+1)2^{-M_0^d})$, such that the $2^{-M_0^d}$ neighbourhood of the subset
\begin{equation}\label{eqn:first_translation}
    \bigcup_{a\in I_0}\{x+w_1(a):x\in S_a\}\cap \Gamma_1
\end{equation}
has Lebesgue measure $\lesssim M_0^{-1-\alpha}$.

\subsection{Step 2}
Now we fix such an interval $I_1$ of length $2^{-M_0^d}$. After the first step of translations $w_1$, the graph of $x\mapsto f_2(a,x)$ becomes that of $x\mapsto f_2(a,x+u)+v$ for some suitable $u,v$ of the order $O(2^{-M_0^d/3})$. We apply Proposition \ref{prop:f(a,x)_delta} to this $f_2(a,x+u)+v$ with $\delta_0=2^{-M_0^d}$ and $M$ replaced by $M_0$. This effectively gives a translation function $w_2(a)\in \R^{d+1}$ of the order $O(2^{-M_0^d-M_0^d/3})$ that is constant on intervals $I_2$ of the form $(k2^{-2M_0^d},(k+1)2^{-2M_0^d})$, such that the $2^{-2M_0^d}$ neighbourhood of the subset
\begin{equation*}
    \bigcup_{a\in I_1}\{x+w_1(a)+w_2(a):x\in S_a\}\cap \Gamma_2
\end{equation*}
has Lebesgue measure $\lesssim 2^{-M_0^d}M_0^{-1-\alpha}$. Taking union over all $I_1\sub I_0$, we see that the $2^{-2M_0^d}$ neighbourhood of
\begin{equation*}
    \bigcup_{a\in I_0}\{x+w_1(a)+w_2(a):x\in S_a\}\cap \Gamma_2
\end{equation*}
has Lebesgue measure $\lesssim M_0^{-1-\alpha}$. Combining this and \eqref{eqn:first_translation} and using $2^{-M_0^d}\ll M_0^{-1-\alpha}$, we have that the $2^{-2M_0^d}$ neighbourhood of
\begin{equation*}
    \bigcup_{a\in I_0}\{x+w_1(a)+w_2(a):x\in S_a\}\cap (\Gamma_1\cup \Gamma_2)
\end{equation*}
has Lebesgue measure $\lesssim M_0^{-1-\alpha}$.

\subsection{Step $N$}
At Step 3, we fix an interval $I_2$ and apply Proposition \ref{prop:f(a,x)_delta} to $f_3(a,x+u')+v'$ for some suitable tiny $u',v'$ with $\delta_0=2^{-2M_0^d}$ and $M$ replaced by $M_0$. We perform the above process for $N$ times, arriving at a final subset of the form
\begin{equation*}
    E:=\bigcup_{a\in I_0}\{x+w(a):x\in S_a\},
\end{equation*}
where $w(a):=\sum_{i=1}^{N}w_i(a)$ is of the order $O(2^{-N^{-1}M^d/3})$, and is piecewise constant on intervals of the form $(A+j2^{-M^d}(B-A),A+(j+1)2^{-M^d}(B-A))$, $0\le j\le 2^{M^d}-1$.

Lastly, we see that the $2^{-NM_0^d}=2^{-M^d}$ neighbourhood of $E$ has Lebesgue measure $\lesssim M_0^{-1-\alpha}\sim M^{-1-\alpha}$, since $N\sim 1$ and $\cup_{1\le i\le N} \Gamma_i=\R^{d+1}$. This finishes the proof.

 \bibliographystyle{alpha}
 \bibliography{sample}

\begin{thebibliography}{HKLO23}

\bibitem[Bes28]{Besicovitch}
A.~S. Besicovitch.
\newblock On {K}akeya's problem and a similar one.
\newblock {\em Math. Z.}, 27(1):312--320, 1928.

\bibitem[Bes63]{Besicovitch1963}
A.~S. Besicovitch.
\newblock The {K}akeya problem.
\newblock {\em Amer. Math. Monthly}, 70:697--706, 1963.

\bibitem[BR68]{BesicovitchRado}
A.~S. Besicovitch and R.~Rado.
\newblock A plane set of measure zero containing circumferences of every radius.
\newblock {\em J. London Math. Soc.}, 43:717--719, 1968.

\bibitem[CC19]{ChangCsornyei}
A.~Chang and M.~Cs\"ornyei.
\newblock The {K}akeya needle problem and the existence of {B}esicovitch and {N}ikodym sets for rectifiable sets.
\newblock {\em Proc. Lond. Math. Soc. (3)}, 118(5):1084--1114, 2019.

\bibitem[CGY23]{ChenGuoYang2023}
M.~Chen, S.~Guo, and T.~Yang.
\newblock A multi-parameter cinematic curvature.
\newblock 2023.
\newblock arXiv:2306.01606.

\bibitem[C{\'{o}}r93]{CordobaKakeyalowerbound}
A.~C{\'{o}}rdoba.
\newblock The fat needle problem.
\newblock {\em Bull. London Math. Soc.}, 25(1):81--82, 1993.

\bibitem[CYZ23]{ChenYanZhong}
X.~Chen, L.~Yan, and Y.~Zhong.
\newblock On the generalized {H}ausdorff dimension of {B}esicovitch sets.
\newblock 2023.
\newblock arXiv:2304.03633.

\bibitem[Dav71]{DaviesKakeya}
R.~O. Davies.
\newblock Some remarks on the {K}akeya problem.
\newblock {\em Proc. Cambridge Philos. Soc.}, 69:417--421, 1971.

\bibitem[Dav72]{Davies}
R.~O. Davies.
\newblock Another thin set of circles.
\newblock {\em J. London Math. Soc. (2)}, 5:191--192, 1972.

\bibitem[HKLO23]{MR4534746}
S.~Ham, H.~Ko, S.~Lee, and S.~Oh.
\newblock Remarks on dimension of unions of curves.
\newblock {\em Nonlinear Anal.}, 229:Paper No. 113207, 14, 2023.

\bibitem[HRZ22]{KakeyaRdHRZ}
J.~Hickman, K.~M. Rogers, and R.~Zhang.
\newblock Improved bounds for the {K}akeya maximal conjecture in higher dimensions.
\newblock {\em Amer. J. Math.}, 144(6):1511--1560, 2022.

\bibitem[Kei99]{Keich}
U.~Keich.
\newblock On {$L^p$} bounds for {K}akeya maximal functions and the {M}inkowski dimension in {${\bf R}^2$}.
\newblock {\em Bull. London Math. Soc.}, 31(2):213--221, 1999.

\bibitem[Kin68]{Kinney}
J.~R. Kinney.
\newblock A thin set of circles.
\newblock {\em The American Mathematical Monthly}, 75(10):1077--1081, 1968.

\bibitem[KW99]{Kolasa_Wolff}
L.~Kolasa and T.~Wolff.
\newblock On some variants of the {K}akeya problem.
\newblock {\em Pacific J. Math.}, 190(1):111--154, 1999.

\bibitem[KZ19]{KatzZahlKayeka}
N.~H. Katz and J.~Zahl.
\newblock An improved bound on the {H}ausdorff dimension of {B}esicovitch sets in $\mathbb{R}^3$.
\newblock {\em J. Amer. Math. Soc.}, 32(1):195--259, 2019.

\bibitem[Per28]{Perron}
O.~Perron.
\newblock \"{U}ber einen {S}atz von {B}esicovitsch.
\newblock {\em Math. Z.}, 28(1):383--386, 1928.

\bibitem[PYZ22]{PYZ2022}
M.~Pramanik, T.~Yang, and J.~Zahl.
\newblock A {F}urstenberg-type problem for circles, and a {K}aufman-type restricted projection theorem in $\mathbb{R}^3$.
\newblock {\em Amer. J. Math. (to appear)}, 2022.
\newblock arXiv:2207.02259.

\bibitem[Saw87]{Sawyer}
E.~Sawyer.
\newblock Families of plane curves having translates in a set of measure zero.
\newblock {\em Mathematika}, 34(1):69--76, 1987.

\bibitem[Sch62a]{Schoenberg1}
I.~J. Schoenberg.
\newblock On certain minima related to the {B}esicovitch-{K}akeya problem.
\newblock {\em Mathematica (Cluj)}, 4(27):145--148, 1962.

\bibitem[Sch62b]{Schoenberg2}
I.~J. Schoenberg.
\newblock On the {B}esicovitch-{P}erron solution of the {K}akeya problem.
\newblock In {\em Studies in mathematical analysis and related topics}, volume~IV of {\em Stanford Studies in Mathematics and Statistics}, pages 359--363. Stanford Univ. Press, Stanford, CA, 1962.

\bibitem[Sog91]{Sogge_first_cinematic}
C.~D. Sogge.
\newblock Propagation of singularities and maximal functions in the plane.
\newblock {\em Invent. Math.}, 104(2):349--376, 1991.

\bibitem[Tal80]{Talagrand}
M.~Talagrand.
\newblock Sur la mesure de la projection d'un compact et certaines familles de cercles.
\newblock {\em Bull. Sci. Math. (2)}, 104(3):225--231, 1980.

\bibitem[Wis04]{Wisewell}
L.~Wisewell.
\newblock Families of surfaces lying in a null set.
\newblock {\em Mathematika}, 51(1-2):155--162, 2004.

\bibitem[Wol97]{Wolff_Kakeya_L3}
T.~Wolff.
\newblock A {K}akeya-type problem for circles.
\newblock {\em Amer. J. Math.}, 119(5):985--1026, 1997.

\bibitem[WZ22]{WangZahlSticky}
H.~Wang and J.~Zahl.
\newblock Sticky {K}akeya sets and the sticky {K}akeya conjecture.
\newblock 2022.
\newblock arXiv:2210.09581.

\bibitem[WZ24]{WangZahlAssouad}
H.~Wang and J.~Zahl.
\newblock The {A}ssouad dimension of {K}akeya sets in $\mathbb {R}^3$.
\newblock 2024.
\newblock arXiv:2401.12337.

\bibitem[YZ24]{YangZhong}
T.~Yang and Y.~Zhong.
\newblock Construction of a curved {K}akeya set.
\newblock 2024.
\newblock arXiv:2408.01917.

\bibitem[Zah21]{KakeyaRdZahl}
J.~Zahl.
\newblock New {K}akeya estimates using {G}romov's algebraic lemma.
\newblock {\em Adv. Math.}, 380:Paper No. 107596, 42, 2021.

\bibitem[Zah23]{Zahl2023}
J.~Zahl.
\newblock On maximal functions associated to families of curves in the plane.
\newblock 2023.
\newblock arXiv:2307.05894.

\end{thebibliography}

 \end{document}